\documentclass[a4paper,10pt]{article}
\usepackage[english]{babel}
\usepackage[left=44mm,right=40mm,top=60mm,bottom=52mm]{geometry}
\usepackage[symbol]{footmisc}
\usepackage{amsmath}
\usepackage{amssymb}
\usepackage{amsthm}
\usepackage{indentfirst}

\newtheorem{theorem}{Theorem}

\newtheorem{lemma}{Lemma}

\newtheorem{remark}{Remark}

\begin{document}
	\section*{On some generalizations of the property $B$-problem \\ of an $n$-uniform hypergraph\footnote[1]{This work was supported by the Russian Foundation for Basic Research under grant 18-01-00355 and by the program "Leading Scientific Schools" under grant NSh-6760.2018.1.}}
	\begin{flushright}
		\textbf{YU. A. DEMIDOVICH}
		
		{\it Moscow Institute of Physics and Technology\\
			yuradem9595@mail.ru}
	\end{flushright}
	{\small \textbf{Key words:} Uniform hypergraphs, property $B,$ simple hypergraphs.
		
		{\small \begin{center}
				\textbf{Abstract}
			\end{center}

			The extremal problem of hypergraph colorings related to Erd{\H{o}}s--Hajnal property $B$-problem is considered. Let $k$ be a natural number. The problem is to find the value of $m_k(n)$ equal to the minimal number of edges in an $n$-uniform hypergraph not admitting $2$-colorings of the vertex set such that every edge of the hypergraph contains at least $k$ vertices of each color. In this paper we obtain new lower bounds for $m_k(n).$}}
		
		\section{Introduction and history}
		
		In this work we consider a famous problem concerning vertex colorings of uniform hypergraphs. First we will give some basic definitions. A \textit{hypergraph} is a pair $H = \left( V, E \right),$ where $V$ is a finite set whose elements are called \textit{vertices} and $E$ is a family of subsets of $V,$ called the \textit{edges}. A hypergraph is said to be $n$-\emph{uniform} if each of its edges contains exactly $n$ vertices.
		
		One of the classical extremal problems of a hypergraph theory is the property $B$-problem. We say that a hypergraph has property $B$ if there is a two-coloring of $V$ such that no edge is monochromatic. The problem is to find the quantity $m(n)$ which is the minimum possible number of edges of an $n$-uniform hypergraph that does not have property $B$. This question was first posed in the paper of Erd\H{o}s and Hajnal (see \cite{EH}). Erd{\H{o}}s himself (see \cite{E1}, \cite{E2}) obtained the following asymptotical bounds for $m(n):$
		\begin{equation}
		2^{n-1} \leqslant m(n) \leqslant \frac{e \ln2}{4}n^2 2^n (1 + o(1)).
		\end{equation}
		
		The upper bound remains the same. The lower bound was refined in several works (see surveys \cite{K}, \cite{RaiSh}). The best-known result is due to Radhakrishnan and Srinivasan (see \cite{RS}). They proved that
		$$
		m(n) \geqslant (0.1)2^n \sqrt{\frac{n}{\ln n}}.
		$$
		
		There exist a lot of different generalizations of the Erd{\H{o}}s--Hajnal problem (see \cite{K}). One of them was suggested by Shabanov (see \cite{Sh1}). Let $k$ be a natural number. We say that a hypergraph $H = \left( V, E \right)$ has \textit{property} $B_k,$ if there exists such a two-coloring  of $V$ that every edge contains at least $k$ vertices of each color. Similarly to $m(n)$ we define the quantity $m_k(n)$ equal to the minimum number of edges of an $n$-uniform hypergraph that does not have property $B_k.$ Clearly property $B_1$ is property $B$ and $m_1(n)$ is $m(n).$ By means of the  probabilistic method one can obtain the following lower bound for $m_k(n)$, which coincides with $\left( 1 \right)$ in case of $k = 1:$
		$$
		m_k(n) \geqslant \frac{2^{n-1}}{\sum_{j = 0}^{k - 1}\binom{n}{j}}.
		$$
		In works \cite{Sh1}, \cite{Sh2} Shabanov proved that if $k = k(n) = o(n/ \ln n),$ then there exists a function $\psi(n),$ which depends only on the type of the function $k(n)$ and tends to $1$ as $n$ approaches infinity, such that
		\begin{equation}
		m_k(n) \leqslant \frac{e\ln 2}{4}n^2 \frac{2^n}{\sum_{i = 0}^{k - 1}\binom{n}{i}}\psi(n).
		\end{equation}
		Moreover in \cite{T} Teplyakov proved that if $k = o(n),$ then $m_k(n)$ satisfies $\left( 2 \right).$
		
		Also in \cite{Sh3} Shabanov obtained the following lower bound for $m_k(n):$
		\begin{equation}
		m_k(n) = \Omega\left(\left(\frac{n}{\ln n}\right)^{1/2}\frac{e^{-k/2}}{\sqrt{2k-1}} \frac{2^{n-k}}{\binom{n}{k-1}}\right)
		\end{equation}
		when $k = k(n) = O(\ln n).$ In case when the growth of $k$ is between $\ln n$ and $\sqrt{n}$ the best lower bound is due to Rozovskaya (see \cite{R1}, \cite{R2}):
		\begin{equation}
		m_k(n) \geqslant 0.19\cdot n^{1/4}\frac{2^{n-1}}{\binom{n-1}{k-1}}.
		\end{equation}
		Recent results and surveys on this topic can be found in the papers
		\cite{RaiSh}, \cite{AkSh} --- \cite{Sh6}.
		\section{Main result}
		The main result of this paper is a new lower bound for $m_k(n).$
		\begin{theorem}
			Let $n \geqslant 30,$ $k\geqslant 2$ and
			\begin{equation}
			k \leqslant \sqrt{\frac{n}{\ln n}}.
			\end{equation}
			Then
			\begin{equation}
			m_k(n) \geqslant \frac{12}{e^{26}} \left(\frac{n}{k\ln n} \right)^{1/2} \frac{2^{n-1}}{\binom{n-1}{k-1}}.
			\end{equation}
		\end{theorem}
		
		The estimate $(6)$ refines the previous results $(3)$ and $(4)$ when the growth of $k$ does not exceed $\frac{n^{1/2 - \delta}}{\sqrt{\ln n}}$ where $0 < \delta < 1/2.$
		
		\section{Proof of Theorem 1}
		In order to prove this theorem we need to show that any $n$-uniform hypergraph with at least $12/e^{26} \times \left(n/k\ln n\right)^{1/2} \times 2^{n-1}/\binom{n-1}{k-1}$ edges has property $B_k.$ Consider an $n$-uniform hypergraph $H = \left( V, E \right) $ with
		\begin{equation}
		|E| = m < \frac{12}{e^{26}} \left(\frac{n}{k\ln n} \right)^{1/2} \frac{2^{n-1}}{\binom{n-1}{k-1}}.
		\end{equation}
		
		\subsection{Criterea of property $B_k$}
		
		Let $G = \left(W, U \right)$ be an arbitrary hypergraph and let $\sigma$ be a numeration of its vertices. More precisely, let $\sigma$ be some bijective mapping from $W$ to $\left\lbrace 1,2,\ldots, |W|\right\rbrace.$ We say that a pair of edges $\left(A_1, A_2 \right)$ of $G$ forms an \textit{ordered} \textit{2-chain} with respect to numeration $\sigma$ if   $|A_1\cap A_2 | = 1$ and $\sigma(v) \leqslant \sigma(u)$ for all $v\in A_1, u \in A_2.$ Pluh\'ar in his work (see \cite{P}) suggested a lemma which connects together the existence of ordered $2$-chains in a hypergraph and property $B.$
		\begin{lemma}
			Let $G = \left(W,U\right)$ be an arbitrary $n$-uniform hypergraph. Then $G$ has property $B$ if and only if there is such a numeration $\sigma$ of $W$ that there are no ordered $2$-chains in $G.$
		\end{lemma}
		
		Further we consider the criterion for having property $B_k$, which generalizes a property $B$-criterion and was formulated by Rozovskaya in \cite{R1}.
		
		For each edge $f \in U$, let $F_{\sigma}(f)$ denote the set of the first $k$ vertices of $f$ in the numeration $\sigma$ and let $L_{\sigma}(f)$ denote the set of the last $k$ vertices of $f$.
		
		\begin{lemma}
			Let $G = \left(W, U\right)$ be an arbitrary $n$-uniform hypergraph every edge of which contains at least $2k$ vertices. Then $G$ has property $B_k$ if and only if there is such a numeration of its vertices $\sigma$ that for each two edges $f$ and $s$, the following relation holds
			\begin{equation}
			L_{\sigma}(f) \cap F_{\sigma}(s) = \varnothing.
			\end{equation}
		\end{lemma}
\begin{proof}		
		\textbf{Necessity.} Suppose that $G$ have property $B_k$, i.e. there exists a coloring of its vertices with two colors such that each edge contains at least $k$ vertices of each color. We enumerate the vertices of a hypergraph in the following way: we assign first numbers to the vertices of the first color and after that we enumerate the vertices of the second color. It is easy to see that the condition $L_{\sigma}(f) \cap F_{\sigma}(s) = \varnothing$ holds true.
		
		\textbf{Sufficiency.} Consider the numeration $\sigma$ given by the theorem. We color the first $k$ vertices of every edge in the numeration $\sigma$ with the first color and all the remaining vertices with the second one. Since the condition $\left( 8\right)$ holds every edge contains at least $k$ vertices of each color. Therefore $G$ has property $B_k.$
\end{proof}

		\subsection{General idea of proof}
		
		Here we generalize the idea of Cherkashin and Kozik (see $\cite{CK}$). They suggested to enumerate $V$ in the following way. For every vertex $v$, let $X_v$ be an random variable with a uniform distribution on $\left( 0, 1\right).$ The value $\sigma(v)$ is defined as follows:
		$$
		\sigma(v) := \sum_{w \in V}I\left\lbrace X_w \leqslant X_v\right\rbrace.
		$$
		For any edge $A \in E$, we define
		$$
		max(A) = \max_{v \in A}X_v,
		$$
		$$
		min(A)  = \min_{v \in A}X_v.
		$$
		We say that $A$ is \textit{"dense"} if
		$$
		max(A) - min(A) \leqslant \frac{1 - p}{2},
		$$
		where $p=\ln n / n.$
		
		According to the criterion a hypergraph has property $B$ if there are no pairs of edges with one common vertex in the intersection which is the first in the one edge and the last in the other one. If we show that the sum of probabilities of the event that there are dense edges and the event that there exists a pair of edges intersecting in the above way (assuming that all the edges are not dense) is strictly less than $1,$ then with positive probability there exists such a numeration that the graph has property $B.$ It is easy to see that
		$$
		|E| \left[\left(\frac{1-p}{2}\right)^n  + n\left(\frac{1-p}{2} \right)^{n-1}\left(1 - \frac{1-p}{2}\right)  \right] + |E|^2\int_{\frac{1-p}{2}}^{\frac{1+p}{2}}x^{n-1}(1-x)^{n-1} dx< 1
		$$
		if
		$$
		|E| \geqslant c \cdot 2^n \cdot \sqrt{\frac{n}{\ln n}},
		$$
		where $c$ is some universal constant. Thus we obtain the best-known lower bound for $m(n).$
		
		We generalize this idea.
		
		Let $f \in E$ be an edge of $H.$ For each vertex $v \in f$, we introduce events $F_f(v)$ and $L_f(v)$. The event $F_f(v)$ occurs if the vertex $v$ is one of the first $k$ vertices of the edge $f$ and the event $L_f(v)$ occurs if $v$ is one of the last $k$ vertices of $f.$ For every two edges $f$ and $s$, such that $f \cap s \neq \varnothing$ and for any vertex $v_0 \in f \cap s$, denote
		\begin{equation}
		\mathcal{M}(f, s, v_0) = L_f(v_0) \cap F_s(v_0).
		\end{equation}
		Let $\mathcal{M}(f,s),$ be the union of the events $\mathcal{M}(f, s, v_0)$ over all vertices $v_0$ from the intersection of $f$ and $s$:
		$$
		\mathcal{M}(f,s) = \bigcup_{v_0 \in f \cap s}\mathcal{M}(f, s, v_0).
		$$
		
		Consider a numeration $\sigma$ of $V.$ Let $A \in E$ be an edge of $H.$ Without loss of generality we assume that $A = \left\lbrace 1,2,\ldots, n\right\rbrace.$ For any $v \in A$, we consider the random variable $X_v$ and arrange them all in the ascending order: $\left\lbrace X_{(1)}, \ldots, X_{(n)} \right\rbrace.$ Denote
		$$
		l(A) = X_{(n-k+1)},
		$$
		$$
		f(A) = X_{(k)}.
		$$
		
		We call an edge $A$ \textit{dense} if
		$$
		l(A) - f(A) \leqslant \frac{1-p}{2},
		$$
		where $p = 2k \ln n / n.$
		For each edge $A$, let $\mathcal{N}(A)$ denote the event that the edge $A$ is dense. Put
		$$
		\mathcal{R} = \bigcup_{A}\mathcal{N}(A).
		$$
		If we prove that the sum of probabilities of the events
		$\mathcal{R}$ and $\displaystyle \mathcal{F} := \bigcup_{f,s}\mathcal{M}(f,s)$ is strictly less than $1,$ assuming that all edges in $\mathcal{F}$ are not dense then with positive probability there exists such a numeration $\sigma$ of $V$ that for each $f$ and $s$, 
		$$
		L_{\sigma}(f) \bigcap F_{\sigma}(s) = \varnothing,
		$$
		which implies that $H$ has property $B_k$ according to the criterion.
		
		Our next step is to estimate the probabilities of the events $\mathcal{R}$ and $\mathcal{M}(f,s).$
		
		\subsection{The estimate for the probability of $\mathcal{R}$}
		
		\begin{multline}
		\mathbb{P}\left( \mathcal{R}\right)  = \mathbb{P} \left( \bigcup_{A} \mathcal{N}(A)\right) \leqslant |E|\max_{A}\mathbb{P}\left(\mathcal{N}(A) \right) = |E|\mathbb{P}\left(l(A) - f(A) \leqslant \frac{1-p}{2}\right) = \\
		=|E|\left[\mathbb{P}\left(l(A) \leqslant \frac{1-p}{2} \right) + \mathbb{P}\left( l(A) - f(A) \leqslant \frac{1-p}{2} , \text{ } l(A) > \frac{1-p}{2}\right)\right]= \\
		=|E|\text{ }\mathbb{P}\left(l(A) \leqslant \frac{1-p}{2} \right) + |E|\text{ }\mathbb{P}\left( l(A) - f(A) \leqslant \frac{1-p}{2} , \text{ }  l(A) > \frac{1-p}{2}\right).
		\end{multline}
		
		Then we estimate both summands of $(10)$ separately.
		
		\begin{multline}
		|E|\text{ }\mathbb{P}\left( l(A) \leqslant \frac{1-p}{2}\right) = |E|n \binom{n-1}{k-1}\int_{0}^{\frac{1-p}{2}}x^{n-k}(1-x)^{k-1}dx=|E|n\binom{n-1}{k-1} \times \\
		\left. \times\left(\frac{x^{n-k+1}\left( 1-x\right)  ^{k-1}}{n-k+1}\right|_0^{\frac{1-p}{2}}+\frac{k-1}{n-k+1}\int_{0}^{\frac{1-p}{2}}x^{n-k+1}(1-x)^{k-2}dx\right)=\\
		\left. = |E|n\binom{n-1}{k-1}\left(\frac{x^{n-k+1}\left( 1-x\right)  ^{k-1}}{n-k+1}\right|_0^{\frac{1-p}{2}}\left.+\frac{k-1}{n-k+1}\frac{x^{n-k+2}(1-x)^{k-2}}{n-k+2}\right|_0^{\frac{1-p}{2}}\right.+\\
		\left.+\frac{k-1}{n-k+1}\frac{k-2}{n-k+2}\int_{0}^{\frac{1-p}{2}}x^{n-k+2}(1-x)^{k-3}dx\right)=\ldots =\\
		=\left. |E|n\binom{n-1}{k-1}\left(\frac{x^{n-k+1}\left( 1-x\right)  ^{k-1}}{n-k+1}\right|_0^{\frac{1-p}{2}}\left.+\frac{k-1}{n-k+1}\frac{x^{n-k+2}(1-x)^{k-2}}{n-k+2}\right|_0^{\frac{1-p}{2}}\right.+\\
		\left.+\ldots+\left.\frac{k-1}{n-k+1}\frac{k-2}{n-k+2}\ldots\frac{1}{n-1}\frac{x^n}{n}\right|_0^{\frac{1-p}{2}}\right)=\\
		= |E|n\binom{n-1}{k-1}\left( \frac{\left( \frac{1-p}{2}\right)^{n-k+1}\left(1-\frac{1-p}{2} \right)^{k-1}}{n-k+1}+\frac{k-1}{n-k+1}\frac{\left(\frac{1-p}{2}\right)^{n-k+2}}{n-k+2}+\right.\\
		\left.+\ldots+\frac{k-1}{n-k+1}\frac{k-2}{n-k+2}\ldots\frac{1}{n-1}\frac{\left( \frac{1-p}{2}\right)^n}{n}\right).
		\end{multline}
		Note that
		\begin{multline}
		\nonumber
		\frac{\left(\frac{1-p}{2}\right)^{n-k+1}\left(1 - \frac{1-p}{2} \right)^{k-1}}{n-k+1}\geqslant \frac{k-1}{n-k+1}\frac{\left(\frac{1-p}{2}\right)^{n-k+2}\left(1 - \frac{1-p}{2} \right)^{k-2}}{n-k+2}\\
		\geqslant\ldots\geqslant \frac{k-1}{n-k+1}\frac{k-2}{n-k+2}\ldots\frac{1}{n-1}\frac{\left(\frac{1-p}{2} \right)^n}{n}.
		\end{multline}
		Then (11) can be estimated in the following way:
		\begin{multline}
		|E|n\binom{n-1}{k-1}k\frac{\left(\frac{1-p}{2}\right)^{n-k+1}\left(1 - \frac{1-p}{2}\right)^{k-1}}{n-k+1}<\frac{12}{2e^{26}}\sqrt{\frac{n}{k\ln n}}\frac{nk}{n-k+1}(1-p)^n\times\\\times\left(\frac{1+p}{1-p}\right)^{k-1}
		= \frac{6}{e^{26}}\sqrt{\frac{n}{k\ln n}}\frac{nk}{n-k+1}\left(1 - \frac{\ln n^{2k}}{n} \right)^n\left(1 + \frac{2p}{1-p} \right)^{k-1} \leqslant\\
		\leqslant \frac{6}{e^{26}}\sqrt{\frac{n}{k\ln n}}\frac{nk}{n - k + 1} \frac{1}{n^{2k}}\exp{\left( \frac{2p(k-1)}{1-p}\right) } = S_1.
		\end{multline}
		Here we used the inequalities $\left(1 - \frac{x}{n}\right)^n \leqslant \exp{\left( -x\right)} $ and $1+x \leqslant \exp{(x)}$ for $x > -1.$
		
	Since
		\begin{multline}
		\exp{\left(\frac{2p(k-1)}{1-p}\right)} < \exp{\left(\frac{2pk}{1-p}\right)} = \exp{\left( \frac{2\frac{2k\ln n}{n}k}{1 - \frac{2k\ln n}{n}}\right)} = \exp{\left(\frac{4k^2\ln n}{n - 2k\ln n} \right)} \leqslant\\
		\leqslant \exp{\left(\frac{4\left(\sqrt{\frac{n}{\ln n}} \right)^2 \ln n}{n - 2 \sqrt{\frac{n}{\ln n}}\ln n} \right)} = \exp{\left(\frac{4n}{n - 2 \sqrt{n \ln n}}\right)} = \exp{\left(\frac{4}{1 - 2\sqrt{\frac{\ln n}{n}}} \right)} \leqslant\\
		\leqslant \exp{\left(\frac{4}{1 - 2 \sqrt{\frac{\ln 30}{30}}} \right)} < e^{13}\quad \text{for all}\quad n \geqslant 30,
		\end{multline}
		we obtain
		\begin{multline}
		S_1 < \frac{6}{e^{13}}\sqrt{\frac{n}{k \ln n}}\frac{nk}{n - k + 1}\frac{1}{n^{2k}} = \frac{6}{e^{13}}\sqrt{\frac{n^3k^2}{k \ln n (n-k+1)^2n^{4k}}} \leqslant \\
		\leqslant \frac{6}{e^{13}}\sqrt{\frac{\sqrt{\frac{n}{\ln n}}}{\ln n (n - \sqrt{\frac{n}{\ln n}})^2n^{4k-3}}} \leqslant \frac{6}{e^{13}}\sqrt{\frac{1}{27^2(\ln n)^{3/2}n^{4k - 7/2}}}<\\
		<\frac{6}{27e^{13}}\sqrt{\frac{1}{n \ln n}} < \frac{1}{45e^{13}} \quad \text{for all} \quad n \geqslant 30.
		\end{multline}
		Then we estimate the second summand in $(10).$
		\begin{multline}
		|E|\mathbb{P}\left(l(A) - f(A) \leqslant \frac{1-p}{2}, \text{ } l(A) > \frac{1-p}{2}\right) = \\
		= |E|n \binom{n-1}{k-1} \int_{\frac{1-p}{2}}^{1}\sum_{i = 0}^{k - 1}\binom{n - k}{i}\left(x - \frac{1-p}{2} \right)^i\left(\frac{1-p}{2} \right)^{n - k - i}(1 - x)^{k-1}dx =\\
		=|E|n \binom{n-1}{k-1}\sum_{i = 0}^{k-1}\binom{n-k}{i}\left(\frac{1-p}{2} \right)^{n-k-i}\int_{\frac{1-p}{2}}^{1}\left( x - \frac{1-p}{2}\right)^i(1-x)^{k-1}dx =\\
		=|E|n \binom{n-1}{k-1}\sum_{i = 0}^{k-1}\binom{n-k}{i}\left(\frac{1-p}{2} \right)^{n-k-i}\frac{k-1}{i+1}\times\\ \times\int_{\frac{1-p}{2}}^{1}\left( x - \frac{1-p}{2}\right)^{i + 1}(1-x)^{k-2}dx  
		=\ldots=|E|n \binom{n-1}{k-1}\times \\ \times \sum_{i = 0}^{k-1}\binom{n-k}{i}\left(\frac{1-p}{2} \right)^{n-k-i}\times\frac{\left( k-1\right)!}{\left( i+1\right) \ldots \left(i + k -1 \right)}\int_{\frac{1-p}{2}}^{1}\left( x - \frac{1-p}{2}\right)^{i + k - 1}dx =\\
		=|E|n \binom{n-1}{k-1}\sum_{i = 0}^{k-1}\binom{n-k}{i}\left(\frac{1-p}{2} \right)^{n-k-i}\frac{(k-1)!i!}{(i+k)!}\left.\left(x - \frac{1-p}{2} \right)^{i+k}\right|_{\frac{1-p}{2}}^1 = \\
		=|E| n \binom{n-1}{k-1}\sum_{i = 0}^{k-1}\binom{n-k}{i}\frac{(k-1)!i!}{(i+k)!}\left(\frac{1-p}{2} \right)^{n-k-i}\times \\ \times\left(1 - \frac{1-p}{2}\right)^{i + k} = S_2.
		\end{multline}
		
		We will estimate the inner sum in $(15).$ For this purpose, we will find the maximum term in the sum. The ratio of term $i$ and term $i+1$ is equal to
		\begin{multline}
			\nonumber
			\frac{\binom{n-k}{i}\frac{i!}{(i+k)!}\left(\frac{1+p}{1-p} \right)^i}{\binom{n-k}{i+1}\frac{\left(i+1 \right)!}{\left(i+k+1 \right)!}\left(\frac{1+p}{1-p}\right)^{i+1}} = \frac{(i+k+1)(1-p)}{(n - k - i)(1+p)} < \frac{i+k+1}{n-k-i} \leqslant \\
			\leqslant \frac{2k}{n - 2k} = \epsilon(n,k) = \epsilon < 1.
		\end{multline}
		
		We used the condition $i \leqslant k - 1.$ Hence, the greatest term in the sum has the greatest number $i$ which is $k-1.$ Therefore, the sum can be estimated by the greatest term multiplied by $(1 - \epsilon)^{-1}.$ Then
		\begin{multline}
		S_2 < \left(1 - \epsilon \right) ^{-1} \frac{12}{2e^{26}}\sqrt{\frac{n}{k\ln n}}n\binom{n-k}{k-1}\frac{\left( k-1\right)!\left( k-1\right)!}{\left( 2k-1\right)!}\left( 1 - p\right)^n\left(\frac{1+p}{1-p} \right)^{2k-1} = \\
		= (1 - \epsilon)^{-1}\frac{6}{e^{26}}\sqrt{\frac{n}{k \ln n}}\frac{n\left( n - k \right)! \left( k - 1\right)!}{\left( n-2k+1\right)!\left( 2k - 1\right)!}\left(1 - \frac{\ln n^{2k}}{n}\right)^n\left(1 + \frac{2p}{1-p}\right)^{2k-1} \leqslant\\
		\leqslant (1 - \epsilon)^{-1}\frac{6}{e^{26}}\sqrt{\frac{n}{k \ln n}}n(n - k)^{k-1}\frac{\left(k-1\right)!}{\left(2k-1\right)!}\frac{1}{n^{2k}}\exp{\left(\frac{2p(2k-1)}{1-p}\right)}<\\
		<(1 - \epsilon)^{-1}\frac{6}{e^{26}}\sqrt{\frac{n}{k \ln n}}n^k\frac{(k-1)!}{(2k-1)!}\frac{1}{n^{2k}}e^{26} =\\
		=6(1-\epsilon)^{-1}\sqrt{\frac{n}{k \ln n}}\frac{(k-1)!}{(2k-1)!n^k}<\\
		<6(1-\epsilon)^{-1}\sqrt{\frac{1}{n \ln n}} < \frac{3}{5}(1 - \epsilon)^{-1} \quad \text{for all} \quad n \geqslant 30.
		\end{multline}
		Hence taking into account $(14)$ and $(16)$ we derive the final estimate for the probability of the event $\mathcal{R}.$
		\begin{equation}
		\mathbb{P}\left( \mathcal{R}\right) < \frac{1}{45e^{13}} + \frac{3}{5}(1 - \epsilon)^{-1}.
		\end{equation}
		
		\subsection{The estimate for the probability of $\mathcal{M}(f,s)$}
		Consider two arbitrary edges $f$ and $s.$ Recall that we assume that all edges are not dense. Let $h$ be the number of their common vertices and let $v_0$ be some vertex from $f \cap s.$ Denote
		$$
		A_1 = \left\lbrace v \in f \cap s: \sigma(v) < \sigma(v_0) \right\rbrace, \quad B_1 = \left\lbrace v \in f \cap s: \sigma(v) > \sigma(v_0) \right\rbrace,
		$$
		$$
		A_2 = \left\lbrace v \in f \setminus s: \sigma(v) > \sigma(v_0)\right\rbrace, \quad B_2 = \left\lbrace v \in s \setminus f: \sigma(v) < \sigma(v_0)  \right\rbrace.
		$$
		The event $\mathcal{M}(f,s,v_0)$ implies that
		$$
		|A_2|+|B_1| \leqslant k - 1,
		$$
		$$
		|A_1|+|B_2| \leqslant k - 1.
		$$
		By the definition $(9)$ the event $\mathcal{M}(f,s,v_0)$ is equivalent to an intersection of the following events
		\begin{equation}
		\mathcal{M}(f,s,v_0) = \left\lbrace |A_2| + |B_1| \leqslant k - 1 \right\rbrace \cap \left\lbrace |A_1| + |B_2| \leqslant k - 1\right\rbrace.
		\end{equation}
		We know that $|A_1| + |B_1| = h - 1.$ Consequently, it follows from $\mathcal{M}(f,s,v_0)$ that $h + |A_2|+|B_2| \leqslant 2k - 1.$ If $h \geqslant 2k,$ the latter inequality is impossible as well as the event $\mathcal{M}(f,s,v_0).$ Below we consider only the case of $h \leqslant 2k-1.$
		
		It follows from $(18)$ that
		\begin{equation}
		\mathcal{M}(f,s,v_0) = \bigcup_{\substack{i,j,t \in \mathbb{Z}_+, \\ t + j \leqslant k - 1,\text{ } t \leqslant h - 1 \\ i \leqslant k - h + t}}\left\lbrace|A_2| = i, |B_2| = j, |A_1| = t \right\rbrace.
		\end{equation}
		Recall that
		$$
		\mathcal{M}(f,s) = \bigcup_{v_0 \in f \cap s}\mathcal{M}(f,s,v_0).
		$$
		Then by $(19)$ we derive the estimate for the probability of $\mathcal{M}(f,s):$
		\begin{multline}
		\mathbb{P}\left(\mathcal{M}(f,s)\right) \leqslant h \sum_{t = \max(0, h - k)}^{\min(k-1, h-1)}\binom{h-1}{t}\sum_{i=0}^{k-h+t}\binom{n-h}{i}\sum_{j=0}^{k-1-t}\binom{n-h}{j} \times \\
		\times \int_{\frac{1-p}{2}}^{\frac{1+p}{2}}x^{n-h-i+t+j}(1-x)^{n-h-j+(h-1-t)+i}dx.
		\end{multline}
		
		The limits of the integral in the right-hand side of $(20)$ are from $\frac{1-p}{2}$ to $\frac{1+p}{2}.$ Suppose the opposite: let $X_{v_0} < \frac{1-p}{2}$ or let $X_{v_0}>\frac{1+p}{2}.$
		Consider the case of $X_{v_0} > \frac{1+p}{2}.$ Since $v_0 \in F_{\sigma}(s)$ it follows that $f(s) \geqslant X_{v_0} > \frac{1+p}{2}.$ \\
		Then
		$$
		l(s) - f(s) < l(s) - \frac{1+p}{2} < \frac{1 - p}{2},
		$$
		which contradicts the condition that $s$ is not dense. The case of $X_{v_0} < \frac{1-p}{2}$ is similar to the considered one.
		
		Further we estimate the expression in the right-hand side of $(20).$ First we estimate each integral in $(20).$
		
		\begin{multline}
		\nonumber
		\int_{\frac{1-p}{2}}^{\frac{1+p}{2}}x^{n-h-i+t+j}(1-x)^{n-h-j+(h-1-t)+i}dx = \\
		= \int_{-\frac{p}{2}}^{\frac{p}{2}}\left(\frac{1}{2} + y \right)^{n-h-i+t+j}\left(\frac{1}{2} - y\right)^{n-h-j+(h-1-t)+i}dy =\\
		=\left(\frac{1}{2}\right)^{2n-h-1}\int_{-\frac{p}{2}}^{\frac{p}{2}}\left( 1+2y\right)^{n-h-i+t+j}\left(1-2y \right)^{n-h-j+(h-1-t)+i}dy=\\
		= \left(\frac{1}{2}\right)^{2n-h-1}\int_{-\frac{p}{2}}^{\frac{p}{2}}(1 - 4y^2)^{n-h}(1+2y)^{-i+t+j}(1-2y)^{-j+(h-1-t) + i}dy \leqslant \\
		\leqslant \left(\frac{1}{2}\right)^{2n-h-1}\int_{-\frac{p}{2}}^{\frac{p}{2}}(1+2y)^{-i+t+j}(1-2y)^{-j+(h-1-t) + i}dy  =\\
		= \left(\frac{1}{2}\right)^{2n-h-1}\int_{-\frac{p}{2}}^{\frac{p}{2}}\left(\frac{1-2y}{1+2y}\right)^i \left(\frac{1+2y}{1-2y} \right)^t \left(\frac{1+2y}{1-2y} \right)^j(1-2y)^{h-1}dy \leqslant \\
		\leqslant \left(\frac{1}{2}\right)^{2n-h-1} \left(\frac{1+p}{1-p} \right)^{i+j+t}(1+p)^{h-1}\int_{-\frac{p}{2}}^{\frac{p}{2}}dy \leqslant \\ \leqslant
		p\left(\frac{1}{2}\right)^{2n-h-1}\left(\frac{1+p}{1-p} \right)^{2k}(1+p)^{2k}= p\left(\frac{1}{2}\right)^{2n-h-1}\left(1 + \frac{2p}{1-p} \right)^{2k}(1+p)^{2k} \leqslant \\
		\leqslant p\left(\frac{1}{2} \right)^{2n-h-1}e^{26}\exp{\left(\frac{4k^2\ln n}{n}\right)} \leqslant e^{30}p\left( \frac{1}{2}\right)^{2n-h-1}.
		\end{multline}
		Thus,
		\begin{multline}
		\mathbb{P}\left( \mathcal{M}(f,s)\right)  \leqslant e^{30}h \sum_{t = \max{(0, h-k)}}^{\min{(k-1, h - 1)}}\binom{h-1}{t}\sum_{i=0}^{k-h+t}\binom{n-h}{i}\times\\ \times\sum_{j=0}^{k-1-t}\binom{n-h}{j}p\left( \frac{1}{2}\right)^{2n-h-1}.
		\end{multline}
		
		We estimate the inner sum over $i$ in the right-hand side of $(21).$ We will find the maximum term. The ratio of the term $i$ and the term $i+1$ equals
		$$
		\frac{\binom{n-h}{i}}{\binom{n-h}{i+1}} = \frac{i+1}{n-h-i}\leqslant\frac{k}{n-h-k}\leqslant\frac{k}{n-3k}=\alpha(n,k)=\alpha<1.
		$$
		Here we used the inequalities $i+1 \leqslant k$ and $h \leqslant 2k.$ Hence, the term is maximized when $i$ is maximum namely when $i$ is equal to $k-h+t.$ The sum can be estimated by the maximum term multiplied by $(1-\alpha)^{-1}.$ We have
		
		\begin{multline}
		\mathbb{P}\left( \mathcal{M}(f,s)\right)  \leqslant e^{30}h \sum_{t = \max{(0, h-k)}}^{\min{(k-1, h - 1)}}\binom{h-1}{t}\sum_{j=0}^{k-1-t}\binom{n-h}{j}\times\\\times\binom{n-h}{k-h+t}(1-\alpha)^{-1}p\left(\frac{1}{2}\right)^{2n-h-1}.
		\end{multline}
		
		We estimate the inner sum over $j$ in the right-hand side of $(22)$ in the same way. The ratio of the term $j$ and the term $j+1$ equals
		$$
		\frac{\binom{n-h}{j}}{\binom{n-h}{j+1}} = \frac{j+1}{n-h-j}\leqslant\frac{k}{n-h-k}\leqslant\frac{k}{n-3k}=\alpha(n,k)=\alpha<1.
		$$
		These inequalities are similar to the previous ones for the sum over $i.$
		
		As before the term is maximum when $j=k-1-t.$ The sum can be estimated by the maximum term multiplied by $(1-\alpha)^{-1}.$ Therefore, we have
		\begin{multline}
		\mathbb{P}\left( \mathcal{M}(f,s)\right)  \leqslant e^{30}h \sum_{t = \max{(0, h-k)}}^{\min{(k-1, h - 1)}}\binom{h-1}{t}\binom{n-h}{k-1-t}\binom{n-h}{k-h+t}\times\\\times(1-\alpha)^{-2}p\left(\frac{1}{2}\right)^{2n-h-1}.
		\end{multline}
		
		Consider the general term in the sum of $(23)$ as the function of $h.$ The ratio of the term $h$ and the term $h+1$ equals
		\begin{multline}
		\frac{h}{h+1}\frac{\binom{h-1}{t}\binom{n-h}{k-1-t}\binom{n-h}{k-h+t}2^{h+1}}{\binom{h}{t}\binom{n-h-1}{k-1-t}\binom{n-h-1}{k-h-1+t}2^{h+2}} = \frac{(n-h)^2(h-t)}{2(h+1)(k-h+t)(n-h-k+t+1)}\geqslant\\\geqslant\frac{(n-2k)^2}{2k^2(n-k)} = \beta.
		\end{multline}
		The latter inequality requires clarification. Put $x = h - t,$ then 
		$$
		\frac{(n-h)^2(h-t)}{2(h+1)(k-h+t)(n-h-k+t+1)} = \frac{x}{x+t+1}\cdot\frac{(n-h)^2}{2(k-x)(n-x-k+1)}
		$$
		is an increasing function of $x.$ From the conditions
		$$
		x = h - t \geqslant 1, \quad t \leqslant k - 1, \quad h < 2k
		$$
		we obtain $(24).$
		
		It follows from the condition $(5)$ of the theorem that $\beta$ is strictly greater than $1.$ We have
		\begin{multline}
		\nonumber
		(n-2k)^2-2k^2(n-k) > n^2-4nk-2k^2n \geqslant n^2 - \frac{4n\sqrt{n}}{\sqrt{\ln n}} - \frac{2n^2}{\ln n} =\\ =\frac{n^2}{\ln n}\left(\ln n - 2 - 4\sqrt{\frac{\ln n}{n}} \right) > 0 ,
		\end{multline}
		since
		$$
		\ln n - 2 - 4 \sqrt{\frac{\ln n}{n}} \geqslant \ln30 - 2 - 4\sqrt{\frac{\ln 30}{30}} > 0 \quad \text{for all} \quad n\geqslant 30.
		$$
		Hence, every term of the right-hand side of $(23)$ is a decreasing function of $h.$ So for all $t$, due to the restriction $h \geqslant t + 1$ we have
		$$
		h\binom{h-1}{t}\binom{n-h}{k-1-t}\binom{n-h}{k-h+t}2^{h+1} \leqslant (t+1)\binom{n-t-1}{k-1-t}\binom{n-t-1}{k-1}2^{t+2}.
		$$
		The relation $(23)$ implies that
		\begin{equation}
		\mathbb{P}\left(\mathcal{M}(f,s) \right) \leqslant \frac{e^{30}}{(1-\alpha)^2}\sum_{t=\max(0, h-k)}^{\min(k-1,h-1)}(t+1)\binom{n-t-1}{k-1-t}\binom{n-t-1}{k-1}p\frac{2^{t+2}}{2^{2n}}.
		\end{equation}

		Clearly the greater number of positive terms makes the sum greater, thus
		$$
		\mathbb{P}\left(\mathcal{M}(f,s) \right) \leqslant e^{30}(1-\alpha)^{-2}\sum_{t=0}^{\min(k-1,h-1)}(t+1)\binom{n-t-1}{k-1-t}\binom{n-t-1}{k-1}p\frac{2^{t+2}}{2^{2n}}.
		$$
		
		We will find the maximal term in the sum over $t$ by considering the ratio of neighbouring terms as before. The ratio of the term $t$ and the term $t+1$ is equal to
		\begin{multline}
		\nonumber
		\frac{(t+1)\binom{n-t-1}{k-1-t}\binom{n-t-1}{k-1}2^{t+2}}{(t+2)\binom{n-t-2}{k-2-t}\binom{n-t-2}{k-1}2^{t+3}} = \frac{(t+1)(n-t-1)^2}{2(t+2)(k-1-t)(n-t-k)}\geqslant\\\geqslant \frac{(n-k)^2}{4k(n-k)} = \frac{n-k}{4k} = \gamma^{-1}>1.
		\end{multline}
		Here we used the inequalities $t\geqslant 0$ and $t + 1 \leqslant k.$
		
		Thus the maximum term corresponds to the case $t=0.$ The sum can be estimated by the maximum term multiplied by $(1 - \gamma)^{-1}.$ Finally we obtain this estimate for the probability of the event $\mathcal{M}(f,s):$
		\begin{equation}
		\mathbb{P}\left( \mathcal{M}(f,s)\right) \leqslant e^{30}(1 - \alpha)^{-2}(1-\gamma)^{-1}\binom{n-1}{k-1}^2p\left(\frac{1}{2} \right)^{2n-2}.
		\end{equation}
		
		\subsection{Auxiliary analysis}
		
		In this section we estimate the following factors:
		$$
		R(n,k) = (1 - \alpha)^{-2}(1 - \gamma)^{-1},
		$$
		$$
		T(n, k) = (1 - \epsilon)^{-1},
		$$
		which appear in the right-hand sides of $(17)$ and $(26)$ correspondingly. We appeal to the condition $(5)$ of the theorem.
		
		The following relations hold true:
		$$
		\alpha = \frac{k}{n-3k} \leqslant \frac{\sqrt{\frac{n}{\ln n}}}{n-3\sqrt{\frac{n}{\ln n}}} = \frac{1}{\sqrt{n\ln n} - 3} \leqslant \frac{1}{\sqrt{30\ln 30} - 3} \leqslant 0.15,
		$$
		$$
		\gamma = \frac{4k}{n-k} \leqslant \frac{4\sqrt{\frac{n}{\ln n}}}{n-\sqrt{\frac{n}{\ln n}}} = \frac{4}{\sqrt{n\ln n} - 1} \leqslant \frac{1}{4\sqrt{30\ln 30} - 1} \leqslant 0.5,
		$$
		$$
		\epsilon = \frac{2k}{n - 2k} \leqslant \frac{2\sqrt{\frac{n}{\ln n}}}{n-2\sqrt{\frac{n}{\ln n}}} = \frac{2}{\sqrt{n\ln n} - 2} \leqslant \frac{2}{\sqrt{30\ln 30} - 2} \leqslant 0.25,
		$$
		$$
		(1 - \alpha)^{-2} \leqslant (1 - 0.15)^{-2} \leqslant 1.5,
		$$
		$$
		(1 - \gamma)^{-1} \leqslant (1 - 0.5)^{-1} = 2,
		$$
		$$
		(1 - \epsilon)^{-1} \leqslant (1 - 0.25)^{-1} = 4/3.
		$$
		The above inequalities yield the final upper bounds for the factors $R(n,k)$ and $T(n,k):$
		\begin{equation}
		R(n,k) = (1 - \alpha)^{-2}(1 - \gamma)^{-1} \leqslant 1.5 \times 2 = 3,
		\end{equation}
		\begin{equation}
		T(n, k) = (1 - \epsilon)^{-1} \leqslant 4/3.
		\end{equation}
		
		\subsection{Completion of the proof}
		
		We complete the proof of the theorem. It follows from $(26)$ and $(27)$ that
		\begin{equation}
		\mathbb{P}\left(\mathcal{M}(f,s) \right) \leqslant 3e^{30}\binom{n-1}{k-1}^2 2^{2-2n}\frac{2k\ln n}{n}.
		\end{equation}
		Then the probability of $\mathcal{F}$ can be estimated as follows:
		\begin{multline}
		\mathbb{P}\left( \mathcal{F}\right) = \mathbb{P}\left(\bigcup_{f,s}\mathcal{M}(f,s)\right) \leqslant |E|^2 \max_{f,s}\mathbb{P}\left(\mathcal{M}(f,s)\right)\leqslant\\ \leqslant |E|^2 3e^{30}\binom{n-1}{k-1}^2 2^{2-2n} \frac{2k\ln n}{n} < \frac{864}{e^{22}}.
		\end{multline}
		The latter inequality follows from the condition $(7)$ on the hypergraph $H.$
		
		It follows from $(17)$ and $(28)$ that
		\begin{equation}
		\mathbb{P}\left( \mathcal{R}\right)  < \frac{1}{45e^{13}} + \frac{4}{5}.
		\end{equation}

		Thus from $(30)$ and $(31)$ we obtain the final estimate for the sum of the probabilities of the events $\mathcal{F}$ and $\mathcal{R}:$
		$$
		\mathbb{P}\left(\mathcal{F}\right) + \mathbb{P}\left(\mathcal{R}\right) < \frac{864}{e^{22}} + \frac{1}{45e^{13}} + \frac{4}{5} < 1.
		$$
		
		Let us sum up. We have shown that the sum of the probabilities of the events $\mathcal{R}$ and $\mathcal{F}$ is strictly less than $1,$ assuming that all of the edges in $\mathcal{F}$ are not dense. Hence with positive probability in a random numeration $\sigma$ of $V$ for each pair of edges $f$ and $s$, we have
		$$
		L_{\sigma}(f)\cap F_{\sigma}(s) = \varnothing.
		$$
		Therefore $H$ has the property $B_k$ according to the criterion. We have obtained the desired inequality since the hypergraph $H$ with the condition $(7)$ was chosen arbitrarily:
		$$
		m_k(n) \geqslant \frac{12}{e^{26}} \left(\frac{n}{k \ln n}\right)^{1/2}\frac{2^{n-1}}{\binom{n-1}{k-1}}.
		$$
		Theorem 1 is proved.

		\subsection{Corollary: property $B_{k, \varepsilon}$}
		We say that a subhypergraph $H'$ of a hypergraph $H = (V,E)$ is \emph{spanning} if the set of its vertices is $V$ and the set of its edges $E'$ is a subset of $E$. We say that a hypergraph $H=(V,E)$ has $\emph{property}$ $B_{k,\varepsilon}$ if there is a spanning subhypergraph $H'=(V, E')$ with property
		$B_k$ and with $|E'|\ge(1 - \varepsilon)|E|.$ It is easy to see that $\varepsilon$ is from $0$ to $1.$ Property $B_{k, \varepsilon}$ is equivalent to property  $B_k$ when $\varepsilon = 0.$ Shabanov (see $\cite{Sh2}$, $\cite{Sh3}$) showed that if
		$$
		\varepsilon \ge \left(\sum_{j=0}^{k-1}\binom{n}{j}\right)2^{1-n}
		$$
		then the property $B_{k,\varepsilon}$ is trivial that is an arbitrary  $n$-uniform hypergraph has property $B_{k, \varepsilon}$.
		
		The quantity $m_{k, \varepsilon}(n)$ is the minimum number of edges of $n$-uniform hypergraph which does not have property  $B_{k,\varepsilon}$.
		We have $m_{k,\varepsilon}(n)=m_k(n)$ for $\varepsilon = 0.$ It is easy to show that if $\varepsilon < \frac{1}{m_{k}(n)}$ then $m_{k,\varepsilon}(n)=m_k(n)$.
		
		It turns out that the quantity $ m_{k,\varepsilon}(n) $ makes sense only when
		$$
		\varepsilon \in \left(\frac{1}{m_{k}(n)},\left(\sum_{j=0}^{k-1}\binom{n}{j}\right) \cdot 2^{1-n}\right).
		$$
		Consider the inequality (2). Since $\psi(n)$ converges, it is bounded. Hence, we have
		\begin{equation}
		\varepsilon \in \left(R \cdot \left(\sum_{j=0}^{k-1}\binom{n}{j}\right) \cdot2^{1-n} \cdot n^{-2},
		\left(\sum_{j=0}^{k-1}\binom{n}{j}\right) \cdot 2^{1-n}\right),
		\end{equation}
		where $R$ is some constant.
		
		Shabanov (see $\cite{Sh2},$ $\cite{Sh3}$) proved that under such conditions there exist constants $ c, C $ such that for $k \le C\ln n$, the following inequality holds
		\begin{equation}
		m_{k, \varepsilon}(n) \ge c \cdot \varepsilon \cdot \frac{n}{\ln n} \cdot
		\frac{2^{2n}}{\left(\sum\limits_{j=0}^{k-1} \binom{n}{j}\right)^2} \cdot \frac{2^{-2k}e^{-k}}{2k-1}.
		\end{equation}
		Clearly $(33)$ is similar to $(3)$: it was produced by squaring $(3)$ and multiplying it by $\varepsilon.$
		In the note $\cite{Dem}$ the following result for this quantity was obtained:
		\begin{theorem}
			Let $n \ge 14, k \ge 2$ and let
			$$
			2k^2(n-k)\le(n-2k)^2.
			$$
			Then
			$$
			m_{k,\varepsilon}(n) \ge 0.0361 \cdot \varepsilon \cdot \sqrt{n} \cdot \frac{2^{2n-2}}{\binom{n-1}{k-1}^2}.
			$$
		\end{theorem}

		Evidently this inequality is similar to $(4)$ as the inequality $(33)$ was similar to $(3).$ Its advantage is the same: the range of values of $k$ extends from $\ln n$ to the square root of $ n $.
		
		\begin{theorem}
			Let $k \geqslant 2$ and suppose that for all $n\ge n_1 = 30$,
			$$
			k \leqslant \sqrt{\frac{n}{\ln n}}.
			$$
			Put
			$$
			\mathcal{I} = R \cdot \left(\sum_{j=0}^{k-1}\binom{n}{j}\right) \cdot2^{-n-1} \cdot n^{-2},
			$$
			$$
			\mathcal{J} = \frac{2e^{30}}{9}\cdot\frac{1}{2^nn\ln n}.
			$$
			Let $N_1$ be such a natural number that for all $n\geqslant N_1$,
			$$
			\mathcal{I} \geqslant \mathcal{J}.
			$$
			Then for all $n \geqslant \max(n_1, N_1)$
			$$
			m_{k, \varepsilon}(n) \geqslant  \frac{1}{3e^{35}} \frac{2^{2n-2}}{\binom{n-1}{k-1}^2}\frac{n}{\ln \left( n^{2k}\ln n\right) } \frac{\varepsilon}{2}.
			$$
		\end{theorem}

		The estimate for $m_{k,\varepsilon}(n)$ obtained in Theorem 3 improves the result of Theorem 2 and $(33)$ when the growth of $k$ does not exceed $\frac{n^{1/2 - \delta}}{\sqrt{\ln n}}$ where $0 < \delta < 1/2.$
		$\\$
		
		\begin{proof} Let $X$ denote a random variable equal to a number of bad pairs of edges in $H$ in a random numeration $\sigma$ and let $Y$ be a random variable equal to a number of dense edges in $H$ in a random numeration $\sigma.$ If we show that the sum of expected values of $X$ and $Y$ is less than $\varepsilon|E|$ then we prove that $H$ has property $B_{k, \varepsilon}.$ Indeed, if $\mathbb{E}\left(X + Y\right) < \varepsilon|E|$ then there exists a numeration $\sigma$ such that $X(\sigma) + Y(\sigma)  < \varepsilon |E|.$ We take such $\sigma$ and remove an edge from every bad pair and remove all dense edges from the hypergraph as well. The obtained spanning hypergraph $H'$ has at least $|E|(1 - \varepsilon)$ edges. It has property $B_k$ according to the criterion.
		
		Put $p=\ln \left( n^{2k}\ln n \right) / n.$ We need the following auxiliary results:
		\begin{multline}
		\nonumber
		(1+p)^{2k} \leqslant \exp\left( \frac{2k\ln(n^{2k}\ln n)}{n}\right) = \exp\left(\frac{4k^2\ln n}{n} + \frac{2k\ln\ln n}{n} \right) =\\= e^4 \cdot \left(\frac{2\ln\ln n}{\sqrt{n\ln n}} \right) \leqslant e^4 \cdot \left(\frac{2\ln\ln 30}{\sqrt{30\ln 30}} \right) < e^5,
		\end{multline}
		
		\begin{multline}
		\left(\frac{1+p}{1-p}\right)^{k}=\left(1+\frac{2p}{1-p} \right)^k\leqslant \exp\left(\frac{2k\frac{2k\ln n+\ln\ln n}{n}}{1 -\frac{2k\ln n+\ln\ln n}{n}} \right) \leqslant\\ \leqslant\exp\left(\frac{4+\frac{2\ln\ln n}{\sqrt{n\ln n}}}{1 - 2\sqrt{\frac{\ln n}{n}} - \frac{\ln\ln n}{\sqrt{n\ln n}}} \right)\leqslant \exp\left(\frac{4+\frac{2\ln\ln 30}{\sqrt{30\ln 30}}}{1 - 2\sqrt{\frac{\ln 30}{30}} - \frac{\ln\ln 30}{\sqrt{30\ln 30}}} \right) < e^{15}.
		\end{multline}
		
		For the expected value of the random variable $Y$, we have
		\begin{multline}
		\mathbb{E}(Y) \leqslant |E|\max_{A}\mathbb{P}\left(\mathcal{N}(A) \right) = |E|\mathbb{P}\left(l(A) - f(A) \leqslant \frac{1-p}{2}\right) = \\
		=|E|\left[\mathbb{P}\left(l(A) \leqslant \frac{1-p}{2} \right) + \mathbb{P}\left( l(A) - f(A) \leqslant \frac{1-p}{2} ,\text{ }  l(A) > \frac{1-p}{2}\right)\right]= \\
		=|E|\text{ }\mathbb{P}\left(l(A) \leqslant \frac{1-p}{2} \right) + |E|\text{ }\mathbb{P}\left( l(A) - f(A) \leqslant \frac{1-p}{2} ,\text{ }  l(A) > \frac{1-p}{2}\right).
		\end{multline}
		We estimate separately each of two terms in the right-hand side
		of (35) as before.
		\begin{multline}
		|E|\text{ }\mathbb{P}\left( l(A) \leqslant \frac{1-p}{2}\right) = |E|n \binom{n-1}{k-1}\int_{0}^{\frac{1-p}{2}}x^{n-k}(1-x)^{k-1}dx=\\
		=\ldots=|E|n\binom{n-1}{k-1}k\frac{\left(\frac{1-p}{2}\right)^{n-k+1}\left(1 - \frac{1-p}{2}\right)^{k-1}}{n-k+1}=\\=|E|\frac{(n-1)!}{(k-1)!(n-k)!}\frac{nk}{n-k+1}\frac{(1-p)^n}{2^n}\left(\frac{1+p}{1-p}\right)^{k-1}\leqslant\\
		\leqslant |E|\frac{nk}{n-k+1}\frac{n^{k-1}}{2^n(k-1)!}\frac{1}{n^{2k}\ln n}\left(\frac{1+p}{1-p}\right)^{k}   <\\
		< |E|\frac{k}{k-1}\frac{1}{n-k}\frac{e^{15}}{n^k2^n\ln n} < |E|\frac{e^{15}}{n^k2^n\ln n}.
		\end{multline}
		
		The estimates in $(36)$ can be obtained in the same way as in $(11)$ and $(12).$ Further we estimate the second term in the right-hand side of $(35).$
		
		\begin{multline}
		|E|\mathbb{P}\left(l(A) - f(A) \leqslant \frac{1-p}{2},\text{ } l(A) > \frac{1-p}{2}\right) = \\
		= |E|n \binom{n-1}{k-1} \int_{\frac{1-p}{2}}^{1}\sum_{i = 0}^{k - 1}\binom{n - k}{i}\left(x - \frac{1-p}{2} \right)^i\left(\frac{1-p}{2} \right)^{n - k - i}(1 - x)^{k-1}dx =\\
		=\ldots=|E|n \binom{n-1}{k-1}\sum_{i = 0}^{k-1}\binom{n-k}{i}\left(\frac{1-p}{2} \right)^{n-k-i}\times \\ \times\frac{\left( k-1\right)!}{\left( i+1\right) \ldots \left(i + k -1 \right)}\int_{\frac{1-p}{2}}^{1}\left( x - \frac{1-p}{2}\right)^{i + k - 1}dx =\\
		= |E|n \binom{n-1}{k-1}\sum_{i = 0}^{k-1}\binom{n-k}{i}\frac{(k-1)!i!}{(i+k)!}\left(\frac{1-p}{2} \right)^{n-k-i}\left(1 - \frac{1-p}{2}\right)^{i + k} < \\
		< |E|n\binom{n-1}{k-1}^2 \frac{(k-1)!(k-1)!}{(2k-1)!}\frac{1}{2^n}(1-p)^n\left(\frac{1+p}{1-p}\right)^{2k-1}(1-\epsilon)^{-1} < \\
		< |E|\frac{4}{3}\frac{n^{2(k-1)+1}}{2^nn^{2k}\ln n (2k-1)!}\left(\frac{1+p}{1-p}\right)^{2k} \leqslant |E|\frac{2e^{30}}{9}\frac{1}{2^nn\ln n}.
		\end{multline}
		
		From $(32)$ we have
		$$
		\frac{\varepsilon}{4} \in \left(R \cdot \left(\sum_{j=0}^{k-1}\binom{n}{j}\right) \cdot2^{-n-1} \cdot n^{-2},
		\left(\sum_{j=0}^{k-1}\binom{n}{j}\right) \cdot 2^{-n-1}\right).
		$$
		It follows from this relation and from $(36), (37)$ that if $n > \max(n_1, N)$ then
		$$
		|E|\mathbb{P}\left( l(A) - f(A) \leqslant \frac{1-p}{2}, \text{ } l(A) > \frac{1-p}{2}\right) < |E|\frac{2e^{30}}{9}\frac{1}{2^nn\ln n} < \frac{\varepsilon}{4}|E|.
		$$
		$$
		|E|\text{ }\mathbb{P}\left( l(A) \leqslant \frac{1-p}{2}\right) < |E|\frac{e^{15}}{n^k2^n\ln n} < |E|\frac{2e^{30}}{9}\frac{1}{2^nn\ln n} < \frac{\varepsilon}{4}|E|.
		$$
		
		Recall that $p = \ln \left( n^{2k}\ln n \right) / n.$ Then the inequalities $(34)$ similarly to $(29)$ imply that
		$$
		\mathbb{P}\left(\mathcal{M}(f,s) \right) \leqslant 3e^{35}\binom{n-1}{k-1}^2 2^{2-2n}\frac{\ln \left( n^{2k}\ln n\right) }{n}.
		$$
		
		Then we estimate the expected value of the random variable $X$:
		\begin{multline}
		\nonumber
		\mathbb{E}(X) \leqslant |E|^2 \max_{f,s}\mathbb{P}\left(\mathcal{M}(f,s)\right) \leqslant |E|\left( \frac{1}{3e^{35}} \frac{2^{2n-2}}{\binom{n-1}{k-1}^2}\frac{n}{\ln \left( n^{2k}\ln n\right) } \frac{\varepsilon}{2}\right)\times\\ \times\left( 3e^{35}\binom{n-1}{k-1}^2 2^{2-2n} \frac{\ln \left( n^{2k}\ln n\right) }{n} \right) = \frac{\varepsilon}{2}|E|.
		\end{multline}
		
		So, we derive that
		$$
		\mathbb{E}\left(X + Y\right) < \left(\frac{\varepsilon}{4} + \frac{\varepsilon}{4} + \frac{\varepsilon}{2}\right)|E|  = \varepsilon|E|.
		$$
		
		Theorem 3 is proved.
\end{proof}
		
		\begin{remark}
			Since $m_{k,\varepsilon}(n) \geqslant m_k(n)$ it is easy to see that if
			$$
			\varepsilon \in \left(R \cdot \left(\sum_{j=0}^{k-1}\binom{n}{j}\right) \cdot2^{1-n} \cdot n^{-2},
			S\cdot \left(\sum_{j=0}^{k-1}\binom{n}{j}\right) \cdot 2^{1-n}\cdot n^{-1/2}\right),
			$$
			where $S$ is some constant, the estimate $(6)$ is the best for $m_{k,\varepsilon}(n).$
		\end{remark}
		
		\section{Hypergraphs with restrictions on an intersection of edges}
		Consider the following generalization of property $B_k$-problem.
We say that a hypergraph has \emph{property} $A_h$ if any two edges of the hypergraph do not intersect or have at least $h$ common vertices. We define $m_{k,h}(n)$ as the minimum possible number of edges of a hypergraph which has property $A_h$ but does not have property $B_k.$ For $h \geqslant 2k$, this problem does not make sense. Evidently any hypergraph which has property $A_{2k}$ has property $B_k,$ hence quantity $m_{k,h}(n)$ does not exist when $h\geqslant 2k.$ In nontrivial case the lower bound for this quantity was obtained by Shabanov (see $\cite{Sh2}$ ):
		\begin{equation}
		m_{k,h}(n) \geqslant c\left(\frac{3n}{2h\ln n} \right)^{h/3}\frac{2^{n-1}}{\binom{n}{k-1}},
		\end{equation}
		where $k = O(h\ln n)$ and $h < 2k.$
		
		Further, holds true
		$$
		m_{k,h}(n) \geqslant m_k(n).
		$$
		That is all the lower bounds for $m_k(n)$ hold for $m_{k,h}(n)$ as well as the lower bound $(6).$ Rozovskaya (see $[10]$) obtained the following result:
		\begin{equation}
		m_{k,h}(n) \geqslant \frac{0.19n^{1/4}2^{n-1}}{\sqrt{2^{h-k}(h-k+1)\binom{n-1}{k-1}\binom{n-1}{2k-1-h}}},
		\end{equation}
		provided $2k^2(n-k) < (n - 2k)^2.$ In comparison to the previous result $(38)$ the lower bound of Rozovskaya is valid for a larger range of values of $k.$ The next theorem yields a new lower bound for $m_{k,h}(n).$
		
		\begin{theorem}
			Let $n\geqslant 30,$ $k>2,$ $k < h < 2k$ and suppose $(5)$ holds. Then
			\begin{equation}
			m_{k,h}(n)\geqslant\frac{12}{e^{26}}\left(\frac{n}{k\ln n} \right)^{1/2} \frac{2^{n-1}}{\sqrt{2^{h-k}(h-k+1)\binom{n-1}{k-1}\binom{n-1}{2k-1-h}}}.
			\end{equation}
		\end{theorem}
		It is easy to check that the ratio of the right-hand sides of $(40)$ and $(6)$ under the conditions of Theorem 4 is not less than $\left(c_0n/k\right)^{(h-k)/2}$ where $c_0$ is some universal constant. Thus, $(40)$ estimates quantity $m_{k,h}(n)$  better than $(6).$ Moreover the estimate of $m_{k,h}(n)$ from Theorem 4 refines $(39)$ when the growth of $k$ does not exceed $\frac{n^{1/2 - \delta}}{\sqrt{\ln n}}$ where $0 < \delta < 1/2.$
		
		\begin{proof} The proof is similar to the proof of Theorem 1. It suffices to show that an arbitrary $n$-uniform hypergraph  $H = \left( V, E\right)$ which has property $A_h$ and with
		\begin{equation}
		|E| < \frac{12}{e^{26}}\left(\frac{n}{k\ln n} \right)^{1/2}\frac{2^{n-1}}{\sqrt{2^{h-k}(h-k+1)\binom{n-1}{k-1}\binom{n-1}{2k-1-h}}},
		\end{equation}
		has property $B_k$ as well.
		
		Let $H = \left(V, E\right)$ be such a hypergraph. Consider the random numeration $\sigma$ which was constructed in the proof of Theorem 1 (see Section 3.2). We introduced the event $\mathcal{R}$ in the section where we proved Theorem 1. It implies that there exist dense edges in the hypergraph $H$ in the random numeration $\sigma.$ According to $(12), (13), (15)$ and $(16)$ under the conditions of Theorem 1 (that is the condition (5)) the following inequality holds:
		\begin{multline}
		\mathbb{P}\left(\mathcal{R}\right) < e^{13}|E|\binom{n-1}{k-1}2^{-n}\frac{nk}{n-k+1}\frac{1}{n^{2k}}+\\
		+e^{26}(1-\epsilon)^{-1}|E|\binom{n-1}{k-1}2^{-n}n^k\frac{(k-1)!}{(2k-1)!}\frac{1}{n^{2k}}.
		\end{multline}
		\
		From the condition $(41)$ on the number of edges we have
		\begin{multline}
		\nonumber
		|E|\binom{n-1}{k-1}2^{-n} < \frac{6}{e^{26}}\left(\frac{n}{k\ln n} \right)^{1/2}\sqrt{\frac{\binom{n-1}{k-1}}{2^{h-k}(h-k+1)\binom{n-1}{2k-1-h}}} < \\
		<\frac{6}{e^{26}}\left(\frac{n}{k\ln n}\right)^{1/2} \sqrt{\frac{(2k-1-h)!(n-2k+h)!}{(k-1)!(n-k)!}} < \\
		< \frac{6}{e^{26}}\left(\frac{n}{k\ln n}\right)^{1/2}\sqrt{\left(\frac{n-2k+h}{2k-h}\right)^{h-k}} < \\
		< \frac{6}{e^{26}}\left(\frac{n}{k\ln n}\right)^{1/2}n^{(h-k)/2} < \frac{6}{e^{26}}\left(\frac{n}{k\ln n}\right)^{1/2}n^{k/2}.
		\end{multline}
		
		We used the inequality $k < h < 2k.$ From here we have the final estimate for $(42):$
		\begin{multline}
		\mathbb{P}\left(\mathcal{R}\right) < \frac{6}{e^{13}}\left(\frac{n}{k\ln n} \right)^{1/2} \frac{nk}{n-k+1}\frac{1}{n^{3k/2}}+6(1-\epsilon)^{-1}\left(\frac{n}{k\ln n} \right)^{1/2}\frac{1}{n^{k/2}} <\\
		< \frac{6}{e^{13}}\frac{1}{n - \sqrt{\frac{n}{\ln n}}}\frac{1}{\sqrt{n\ln n}}+6(1- \epsilon)^{-1}\frac{1}{\sqrt{n\ln n}} <
		< \frac{6}{e^{13}}\frac{1}{27}\frac{1}{10} + 6(1-\epsilon)^{-1}\frac{1}{10} = \\=\frac{1}{45e^{13}} + \frac{3}{5}(1 - \epsilon)^{-1} < \frac{1}{45e^{13}}+ \frac{4}{5} \text{ äëÿ âñåõ }  n \geqslant 30.
		\end{multline}
		
		According to $(25)$ under the conditions of Theorem 1 (that is the condition $(5) $) the following inequality holds
		$$
		\mathbb{P}\left( \mathcal{M}(f,s)\right) \leqslant e^{30}(1-\alpha)^{-2} \sum_{t = \max(0, l - k)}^{\min(k-1,l-1)}(t+1)\binom{n-t-1}{k-1-t}\binom{n-t-1}{k-1}2^{t+2}p\left(\frac{1}{2}\right)^{2n},
		$$
		where $l = f \cap s.$ Under the conditions of Theorem 2 $\max(0, l-k) = l - k \geqslant h - k.$ Then we can increase the sum if we put $t$ from $h - k$ to $k - 1.$ We find the maximum term in the sum by considering the ratio of neighbouring terms. Similarly to Theorem 1 we deduce that the maximum term has the number $t = h - k$ and the sum can be estimated by the maximum term multiplied by $(1 - \gamma)^{-1}.$ Therefore
		\begin{multline}
		\nonumber
		\mathbb{P}\left(\mathcal{M}(f,s)\right) \leqslant e^{30}(1-\alpha)^{-2}(1-\gamma)^{-1}(h-k+1)\binom{n-h+k-1}{2k-1-h}\times\\
		\times\binom{n-h+k-1}{k-1}2^{h-k+2}p\left(\frac{1}{2} \right)^{2n}.
		\end{multline}
		
		Further, according to $(27)$ we derive the following estimate for the probability of the event $\mathcal{M}(f,s):$
		$$
		\mathbb{P}\left(\mathcal{M}(f,s)\right) \leqslant 6e^{30}2^{2-2n}(h-k+1)\binom{n-1}{2k-1-h}\binom{n-1}{k-1}2^{h-k}\frac{k\ln n}{n}.
		$$
		Finally the probability of the event $\mathcal{F}$ can be estimated in the following way:
		\begin{multline}
		\nonumber
		\mathbb{P}\left(\mathcal{F}\right) \leqslant \sum_{f,s}\mathbb{P}\left(\mathcal{M}(f,s)\right) <\\
		<  |E|^2 6e^{30}2^{2-2n}(h-k+1)\binom{n-1}{2k-1-h}\binom{n-1}{k-1}2^{h-k}\frac{k\ln n}{n} < \frac{864}{e^{22}}.
		\end{multline}
		It follows from the latter inequality and $(43)$ that
		$$
		\mathbb{P}\left( \mathcal{R}\right)  + \mathbb{P}\left( \mathcal{F}\right) < \frac{1}{45e^{13}} + \frac{4}{5} + \frac{864}{e^{22}} < 1.
		$$
		We have proven that with positive probability in the random numeration $\sigma$ for any edges $f$ and $s$,
		$$
		L_{\sigma}(f)\cap L_{\sigma}(s) = \varnothing.
		$$
		Therefore according to the criterion $H$ has property $B_k$ and since the choice of $H$ is arbitrary we obtain the desired inequality
		$$
		m_{k,h}(n) \geqslant \frac{12}{e^{26}}\left(\frac{n}{k \ln n} \right)^{1/2}\frac{2^{n-1}}{\sqrt{2^{h-k}(h-k+1)\binom{n-1}{k-1}\binom{n-1}{2k-1-h}}}.
		$$
		Theorem 4 is proved.
\end{proof}

		\section{Simple hypergraphs}
		
		Another generalization of the question of Erd{\H{o}}s--Hajnal is connected with simple hypergraphs. We say a hypergraph is  \emph{simple} if any two edges of it have no more than one vertex in the intersection. Let quantity $m^*(n)$ denote the minimum possible number of edges in a simple $n$-uniform hypergraph which does not have property $B.$ Initially the problem of finding $m^*(n)$ was posed by P. Erd{\H{o}}s and L. Lov{\'a}sz, (see $\cite{EL}$). They obtained the following results:
		$$
		c_1\frac{4^n}{n^3} \leqslant m^*(n) \leqslant c_2n^44^n.
		$$
		Further the lower bound was refined by Z. Szab{\'o} (see $\cite{Sz}$) and by A.V. Kostochka and M. Kumbhat (see $\cite{KK}$). They proved that for any $\varepsilon > 0$, there exists an integer $n_0(\varepsilon)$ such that for all $n > n_0(\varepsilon)$,
		$$
		m^*(n) \geqslant 4^n n^{-\varepsilon}.
		$$
		We consider the natural generalization of this problem. Define a quantity $m^*_k(n)$ which is equal to minimum possible number of edges in a simple $n$-uniform hypergraph that does not have property $B_k.$ The best lower bound for this quantity is due to Rozovskaya (see $\cite{R1}$):
		\begin{equation}
		m^*_k(n) \geqslant \frac{(0.19)^2}{8e}\frac{2^{2n-2}}{(n-1)^{3/2}\binom{n-2}{k-1}^2},
		\end{equation}
		where $2k^2(n-k-1) < (n - 2k - 1)^2.$ The following theorem yields a new lower bound for $m^*_k(n).$
		
		\begin{theorem}
			Let $k\geqslant 2$ and suppose that $(5)$ hold. Then for all $n\geqslant 30$,
			$$
			m^*_k(n) \geqslant \frac{25}{18e^{52}}\frac{2^{2(n-1)}}{k(n-1)\ln (n-1)\binom{n-2}{k-1}^2}.
			$$
		\end{theorem}

		Clearly the estimate obtained improves the previous result $(44)$ when the growth of $k$ does not exceed $\frac{n^{1/2 - \delta}}{\sqrt{\ln n}}$ where $0 < \delta < 1/2.$ In order to prove Theorem 5 we need the following theorem which produces a sufficient condition for a hypergraph to have property $B_k.$
		
		\begin{theorem}
			Let $H = \left( V,  E\right)$ be an $n$-uniform hypergraph, every edge of which does not intersect more then $D$ other edges. Let $n \geqslant 30,$ $k \geqslant 2$ and suppose that $(5)$ hold. If
			\begin{equation}
			D \leqslant \frac{5}{3e^{26}}\left(\frac{n}{k\ln n} \right)^{1/2}\frac{2^{n-1}}{\binom{n-1}{k-1}} - 1,
			\end{equation}
			then $H$ has property $B_k.$
		\end{theorem}
		
		\subsection{Proof of Theorem  6}

Let $H = \left( V, E\right) $ be an $n$-uniform hypergraph such that every of its edges intersects no more than $D$ other edges.
		
		Let $\sigma$ be a random numeration as in the proof of Theorem 1 (see Section 3.2). Recall that we introduced the events $\left\lbrace \mathcal{M}(f,s): f, s \in E, f \cap s \neq \varnothing\right\rbrace$ and $\left\lbrace \mathcal{N}\left(A\right): A \in E, A \right.$ --- \textit{dense}$\left. \right\rbrace$. If we prove that
		$$
		\mathbb{P}\left\lbrace \left(\bigcap_{f,s}\overline{\mathcal{M}(f,s)} \right) \bigcap \left(\bigcap_{A} \overline{\mathcal{N}(A)}\right) \right\rbrace > 0,
		$$
		then with positive probability for any two edges $f$ and $s$,
		$$
		L_{\sigma}(f) \cap F_{\sigma}(s) = \varnothing,
		$$
		which means that $H$ has property $B_k.$
		
		Further we need a supplementary theorem proof of which can be found in $[5].$
		
		\begin{theorem}
			Let $A_1,\ldots, A_N$ be events in an arbitrary probabilistic space, let $S_1, \ldots, S_N$ be the subsets of $\left\lbrace 1, \ldots, N \right\rbrace.$ If the following conditions hold
			\begin{enumerate}
				\item $A_i$ is independent of the algebra generated by the events $\left\lbrace A_j, j \notin S_i \right\rbrace,$
				\item $\forall i \in \left\lbrace 1, \ldots, N \right\rbrace$ holds $\sum_{j \in S_i\cup i} \mathbb{P}(A_j) \leqslant 1/4.$
			\end{enumerate}
		then
			$$
			\mathbb{P}\left(\bigcap_{i = 1}^{N}\overline{A_i} \right) > 0.
			$$
		\end{theorem}
		
		Due to the construction of the random numeration $\sigma$ the event $\mathcal{M}(f,s)$ is independent of the algebra generated by events
		$$
		\left\lbrace \mathcal{M}(g, u): g,u \in E, g \cap u \neq \varnothing, (g\cup u) \cap (f \cup s) = \varnothing \right\rbrace,
		$$
		$$
		\left\lbrace \mathcal{N}(Q): Q \in E,  Q \cap (f \cup s) = \varnothing \right\rbrace
		$$
		and the event $\mathcal{N}(A)$ is independent of the algebra generated by events
		$$
		\left\lbrace \mathcal{M}(g, u): g,u \in E, g \cap u \neq \varnothing, (g\cup u) \cap A = \varnothing \right\rbrace,
		$$
		$$
		\left\lbrace \mathcal{N}(Q): Q \in E, Q \cap A = \varnothing \right\rbrace .
		$$
		For the event $\mathcal{M}(f,s)$, consider
		\begin{equation*}
		S(f,s)=
		\begin{cases}
		\left\lbrace \mathcal{M}(g, u): g,u \in E, g \cap u \neq \varnothing, (g\cup u) \cap (f \cup s) \neq \varnothing \right\rbrace,\\
		\left\lbrace \mathcal{N}(Q): Q \in E, Q \cap (f \cup s) \neq \varnothing \right\rbrace ,
		\end{cases}
		\end{equation*}
		and for the event $\mathcal{N}(A)$, consider
		\begin{equation*}
		Z(A)=
		\begin{cases}
		\left\lbrace \mathcal{M}(g, u): g,u \in E, g \cap u \neq \varnothing, (g\cup u) \cap A \neq \varnothing \right\rbrace,\\
		\left\lbrace \mathcal{N}(Q): Q \in E, Q \cap A \neq \varnothing \right\rbrace.
		\end{cases}
		\end{equation*}
		Under the conditions of Theorem 6 the number of elements in $S(f,s)$ does not exceed $4(D+1)^2 + 2(D+1)$ and the number of elements in $Z(A)$ does not exceed $2(D+1)^2 + (D+1).$ From $(29)$ and $(42)$ we have
		$$
		2(D+1)^2\mathbb{P}\left(\mathcal{M}(f,s) \right) + (D+1)\mathbb{P}\left(\mathcal{N}(A)\right) < 4(D+1)^2\mathbb{P}\left(\mathcal{M}(f,s) \right) + 
		$$
		$$
		+2(D+1)\mathbb{P}\left(\mathcal{N}(A)\right)< 4(D+1)^2\left(3e^{30}\binom{n-1}{k-1}^2 2^{2-2n} \frac{2k \ln n}{n} \right) + 2(D+1) \times
		$$
		\begin{multline}
		\nonumber
		\times\left(\binom{n-1}{k-1}2^{-n}\frac{nk}{n-k+1}\frac{e^{13}}{n^{2k}} + (1-\epsilon)^{-1}\binom{n-1}{k-1}2^{-n}n^k\frac{(k-1)!}{(2k-1)!}\frac{e^{26}}{n^{2k}} \right) < \\
		< \frac{200}{3e^{22}} + \frac{1}{162e^{13}} + \frac{2}{9} < \frac{1}{4}.
		\end{multline}
		Here we used the condition $(45)$ on $D.$
		
		Thus, Theorem 7 implies that
		$$
		\mathbb{P}\left\lbrace \left(\bigcap_{f,s}\overline{\mathcal{M}(f,s)} \right) \bigcap \left(\bigcap_{A} \overline{\mathcal{N}(A)}\right) \right\rbrace > 0.
		$$
		Theorem 6 is proved.
		
		\subsection{Proof of Theorem 5} From Theorem 6 we know that if any edge of an $n$-uniform hypergraph intersects no more than $D$ other edges where
		$$
		D \leqslant \frac{5}{3e^{26}}\left(\frac{n}{k \ln n} \right)^{1/2} \frac{2^{n-1}}{\binom{n-1}{k-1}} - 1
		$$
		then it has property $B_k.$ Put
		$$
		X(n) = \frac{5}{3e^{26}}\left(\frac{n}{k \ln n} \right)^{1/2} \frac{2^{n-1}}{\binom{n-1}{k-1}} - 1.
		$$
		Then if $n$-uniform hypergraph does not have property $B_k$ then it has an edge $f_0$ which intersects more than $X(n)$ other edges. Hence there exists such a vertex $v_0 \in f_0$ that
		$$
		\text{deg } v_0 \geqslant \frac{X(n)}{n} + 1.
		$$
		
		Let $H = \left( V, E\right) $ be an $n$-uniform hypergraph that does not have property $B_k.$ Consider an $(n-1)$-uniform hypergraph $H_1 = \left( V, E\right)$ which is obtained from $H$ by deleting a vertex of a maximum degree from each edge. Clearly $H_1$ does not have property $B_k$ as $H$ does not have this property. Then there is a vertex $v_1 \in V$ such that
		$$
		\text{deg}_{H_1}\text{ }v_1 \geqslant \frac{X(n-1)}{n-1}+1.
		$$
		
		Consider all the edges that contain $v_1.$ We removed the vertex with the maximum degree from each such edge. We restore them. The initial hypergraph $H$ is simple hence all restored vertices are different. Let $u_1,\ldots,u_m$ be all these vertices and
		$$
		m \geqslant \text{deg}_{H_1}\text{ }v_1 \geqslant \frac{X(n-1)}{n-1}+1.
		$$
		The degree of every $u_j$ in $H$ is not smaller than  $\text{deg}_{H_1}\text{ }v_1$ and every $u_i$ and $u_j$ have no more than one common edge.
		
		Each restored vertex is contained in at least $Y = X(n-1)/(n-1)+1$ edges. Then the total number of edges in the hypergraph can be estimated as follows:
		$$
		Y + (Y-1) + (Y-2) + \ldots = \frac{Y(Y+1)}{2}.
		$$
		We substitute $Y$ in this expression and obtain
		$$
		|E| \geqslant \frac{Y(Y+1)}{2} \geqslant \frac{1}{2}\left(\frac{X(n-1)}{n-1}+1 \right)^2 \geqslant \frac{25}{18e^{52}}\frac{2^{2(n-1)}}{k(n-1)\ln(n-1)\binom{n-2}{k-1}^2}.
		$$
		Theorem 5 is proved.

	\end{document}